\documentclass[12pt]{amsart}
\usepackage[all]{xy}
\usepackage{amsfonts}
\usepackage{amssymb}

\usepackage[cp1251]{inputenc}

\usepackage[T2A]{fontenc}
\usepackage[cp1251]{inputenc}
\usepackage{amsmath,amssymb,amsthm, latexsym, amscd}

\usepackage{color}

\newtheorem{theorem}{Theorem}[section]
\newtheorem{lemma}[theorem]{Lemma}
\newtheorem{prop}[theorem]{Proposition}

\newtheorem{cor}[theorem]{Corollary}

\theoremstyle{definition}
\newtheorem{definition}[theorem]{Definition}
\newtheorem{remark}[theorem]{Remark}

\newtheorem{ex}[theorem]{Example}

\def\<{\langle}
\def\>{\rangle}

\def\a{\alpha}
\def\b{\beta}

\def\d{\delta}

\def\k{{\bf k}}

\def\N{{\mathbb N}}
\def\R{{\mathbb R}}
\def\Z{{\mathbb Z}}
\def\Q{{\mathbb Q}}

\def\J{{\mathcal J}}
\def\O{{\mathcal O}}

\def\1{\mathbf 1}

\def\N{{\mathbb N}}

\begin{document}

\title{On topological properties of the formal power series substitution group}

\author{I.~Babenko, S.~Bogatyi}

\thanks{Partially supported by the grants NSh 1562.2008.1 и RFSF 06-01-00764.}

\address{Universit\'e Montpellier II CNRS UMR 5149,
Institut de Math\'ematiques et de Mod\'e\-lisation de Montpellier,
Place Eug\`ene Bataillon, B\^at. 9, CC051,
34095 Montpellier  CEDEX 5, France and
Department of Mathematics and Mechanics, Moscow State University (Lomonosov),
119992 Moscow, Russia}
\email{babenko@math.univ-montp2.fr; bogatyi@inbox.ru}

\maketitle

\begin{abstract}
Certain topological properties of the group $\J(\k)$ of all formal one-variable power series
with coefficients in a topological unitary ring $\k$ are considered.
We show, in particular, that in the case when $\k=\Q$ the group $\J(\Q)$
has no continuous bijections into a locally compact group.
In the case when $\k=\Z$ supplied with discrete topology, in spite of the fact
that the group $\J(\Z)$ has continuous bijections into compact groups, it cannot be
embedded into a locally compact group.
In the final part of the paper the compression property for topological groups is considered.
We establish the compressibility of $\J(\Z)$.
\end{abstract}

\maketitle

Let ${\bf k}$ be a commutative ring with identity element.
Further on we always suppose  ${\bf k}$ be a topological ring,
if a topology is not indicated explicitly we suppose it to be discrete.
Consider the set $\J({\bf k})$ of all formal power series in the variable $x$
with coefficients in ${\bf k}$ having the form :
$$
f(x) = x + \alpha_1x^2 + \alpha_2x^3 + ... =x(1 + \alpha_1x + \alpha_2x^2 + ...),
\hskip7pt \alpha_n \in {\bf k} .
\eqno(1)
$$

This set becomes a group under the operation of the substitution of series and then removing the parentheses
and collecting similar terms (cf.~\cite{J54}) : $f \circ g = f\bigl(g(x)\bigr)$.
The general algebraic properties of the group $\J({\bf k})$ were studied in \cite{J54}.
This group, especially in the case of a field of finite characteristic, has been of
high interest during the last 10-15 years.
Without going into details we refer the reader to the survey~\cite{C00}, see also
~\cite{BB08} where the case ${\bf k}=\Z$ is considered and where there are references
to some more papers published later.

The group $\J({\bf k})$ is a topological group in a natural way.
Namely the set iso\-morphism $j:{\bf k}^{\aleph_0}\to \J(\k)$, given by (1),
defines the "product topology" on $\J(\k)$. This topology coincides with
the inverse limit topology of groups $\J^m(\k)$ considered further on.

Even if ${\bf k}$ is not \`a priori a topological ring we, as mentioned above,
can supply it with the discrete topology. This defines the 0-dimensional
inverse limit topology on $\J({\bf k})$, (cf. ~\cite{BB08}).
The latter is the strongest topology on  $\J({\bf k})$ of all
natural topologies but the discrete one.

We assume that a topological group $G$ admits a continuous bijection into a topological
group $H$ if there exists a {\it continuous group-monomorphism} $\phi : G \longrightarrow H$.
What is more, if such a monomorphism $\phi : G \longrightarrow H$ is at the same time a homeomorphism onto its
image we say that the topological group  $G$  is embedded into
the topological group $H$.
From this point of view the group $\J(\Z)$,  although being neither compact nor locally compact,
nevertheless admits a continuous bijection into a compact group.

If we supply $\Z$ with the $p$-adic topology and if we embed $\Z$
into the ring $\Z_{(p)}$ of $p$-adic integers for a prime number $p$ we can obtain
the simplest and most natural continuous bijection of $\J(\Z)$ into a compact group.
The compactness of the ring $\Z_{(p)}$ in the $p$-adic topology implies
the compactness of the group $\J(\Z_{(p)})$, moreover $\J(\Z)$ is
an everywhere dense subgroup therein.

It is easy to see that $\J(\Z)$ is a residually finite group.
Moreover it is approximated by finite $p$-groups for any given prime $p$.
Thus one can obtain some other interesting continuous bijections of $\J(\Z)$
into compact groups, for example into the group $\J(\Z_p)$
where $\Z_p = \Z/p\Z$ is a prime field.

It is known that $\J(\Z_p)$ is a universal pro-$p$-group, that is
every countably based pro-$p$-group can be embedded into $\J(\Z_p)$
(cf., for example, \cite{C00}).
The universality of $\J(\Z_p)$ and the approximability of $\J(\Z)$
by finite $p$-groups easily imply that $\J(\Z)$ admits continuous bijections
into $\J(\Z_p)$.
Unlike the bijection into $\J(\Z_{(p)})$ described above
all such bijections are unfortunately implicit.

One of the principal purposes of this paper is to show that $\J(\Z)$
supplied with the natural topology admits no embeddings into a locally
compact group.
\vskip7pt
\noindent
{\bf Theorem 1.}
{\it Let $\k$ be an infinite discrete commutative ring
and let the group $\J(\k)$ be endowed with the inverse limit topology.
Then $\J(\k)$ does not admit any embedding into a locally compact group}.
\vskip7pt
The following result shows that if we pass from $\Z$ to the residual field $\Q$
the situation becomes much more rigid.
\vskip7pt
\noindent
{\bf Theorem 2.}
{\it Let $\k$ be a topological ring containing the field of rational numbers.
Then the group $\J(\k)$ does not admit any continuous bijection into a locally compact group}.
\vskip7pt
Here we don't make any assumptions about the topology of the ring $\k$,
no topological properties of the inclusion $\Q \subset \k$ are
assumed either.

Finally we note that the problem of embedding into locally compact groups for
$\J(\Z)$ (or $\J(\k)$ for rings of general type)
is interesting from both the view point of topology and that of dynamics,
in particular when studying the amenability of such groups.
See \cite{BB09} for the latter question.

\section{Embeddings of the group $\J(\k)$}

Let the base ring $\k$ be supplied with some topology, the case of the discrete one
is of special interest.
The corresponding inverse limit topology on the group $\J(\k)$ is considered.

Subgroups
$$
\J_n(\k) = \{f(x)\in \J(\k)\big{|} \alpha_1 = \alpha_2=...=\alpha_n = 0 \}, \ \
n = 1, 2, ...
$$
are normal and form a base of neighborhoods of the unity element.
It is natural that $\J_0(\k)=\J(\k)$.
The following sequences of nilpotent groups will also be useful
$$
\J^m(\k) = \J(\k) \big{/}\J_m(\k), \ \ m = 1, 2, ... \hskip3pt ; \hskip3pt
\J_n^m(\k) = \J_n(\k) \big{/}\J_m(\k), \ \ m > n \hskip3pt .
\eqno(2)
$$

We divide the proof of the above results into several simpler steps.
The verification of the following statements is not difficult and we shall omit
some part of the proofs.
\begin{prop}
\label{propk}
The mapping $p\colon \J_{n}(\k)\to \k$, $p(f)= \alpha_{n+1}\in \k$ induces
an isomorphism of the topological groups $\hat p\colon \J_{n}^{n+1}(\k)\to \k$, where $\k$
is only considered as an addition group.
\end{prop}
\begin{prop}
\label{propembed}
If a monomorphism $f\colon G\to H$ of topological groups is a topological embedding,
then for any closed normal subgroup
$H_1\subset H$ the subgroup $G_1=f^{-1}(H_1)$ is a closed normal subgroup in
$G$, and the induced homomorphism of factor groups
$$\hat f\colon G/G_1\to H/H_1$$
is a topological embedding.
\end{prop}
\begin{prop}
\label{propisolatpoint}
If $X\subset Y$ is an everywhere dense subset in a space $Y$ having no
isolated points then $X$ has no isolated points either.
\end{prop}
\begin{proof}
Let $x\in X$ be an isolated point in $X$, that is, $x$ is an open subset of $X$.
Choose such a neighborhood $V_x$ of $x$ in $Y$ that $V_x\cap X=x$.
Since $x$ is not an isolated point in $Y$, there exists a point
$y\in V_x\setminus x\subset Y\setminus X$.
The set $(Y\setminus x)\cap V_x$ is a neighborhood of $y$ and it does not contain any points
in $X$. This contradicts the density of $X$ in $Y$.
\end{proof}
\begin{prop}
\label{propgroupisolatpoint}
If a compact group $G$ has an isolated point, then
it is finite.
\end{prop}
\begin{cor}
\label{corcompactgroup}
A topological group containing an infinite discrete subgroup
 cannot be imbedded into a compact group.
\end{cor}

Theorem 1 immediately follows from the following statement.
\begin{theorem}
\label{teocompactsubstit}
If the additive group of a topological ring $\k$ contains an infinite discrete subgroup,
then the group $\J(\k)$  cannot be embedded into a locally compact group.
\end{theorem}

\begin{proof}
Let $f\colon \J(\k)\to \tilde H$ be a topological embedding and let
$\tilde H$ be a locally compact group.
Consider a compact neighborhood $V\subset \tilde H$ of the identity.
Then $U = f^{-1}(V)\subset \J(\k)$ is a neighborhood of the identity in $\J(\k)$.
So there exists such an index $m$ that $\J_m(\k)\subset U$.
The subset $f(\J_m(\k))\subset V$ is a subgroup, so its closure $H = \overline{f(\J_m(\k))}$
is a compact group.
It is easy to see that the restriction mapping (which we denote by the same symbol)
$f\colon \J_m(\k)\to H$ is a topological embedding.
Since the subgroup $G_1 = \J_{m+1}(\k)\subset \J_m(\k)$
is closed and normal in $\J_m(\k)$, then
the closure of its image $H_1 = \overline{f(G_1)}$ is
a closed and normal subgroup of the compact group $H$.
Since the subgroup $f(\J_m(\k))$
is dense in $H$ we have $f(G_1)=H_1\cap f(\J_m(\k))$.

According to Proposition~\ref{propembed} the induced mapping
of factor groups
$$\hat f\colon \J_m(\k)/\J_{m+1}(\k)\to H /H_1$$
is a topological embedding.
The isomorphism $\J_m(\k)/\J_{m+1}(\k)\simeq \k$ of Proposition~\ref{propk}
together with Corollary~\ref{corcompactgroup} lead to a contradiction.
This completes the proof.
\end{proof}

\section{Homomorphisms of $\J(\Q)$ into compact and locally compact groups}

As we saw in the introduction
there are a lot of different continuous bijections of $\J(\Z)$
not only into locally compact groups but into compact groups as well.
The situation radically changes if we pass from integer coefficients to rational ones.

Now we shall study the structure of (continuous) homomorphisms from $\J(\Q)$
into locally compact groups.
Theorem 2 stated in the introduction immediately follows from the following
stronger result, which we shall prove in this section.
\begin{theorem}
\label{teoQ}
Let $G$ be a locally compact group and
$\phi : \J(\Q) \longrightarrow G$ be a (continuous) homomorphism.
Then there exists a natural number $m$ such that
$\J_m(\Q) \subset \ker \phi$.
If furthermore $G$ is compact then $\J_2(\Q) \subset \ker \phi$.
\end{theorem}
We recall (cf., for example, \cite{KM85}) that a group $G$ is called
(algebraically) {\it complete}
if for any $g \in G$ and any natural $k$
there exists an element $h \in G$ such that $h^k = g$.
This property is also known \cite{Scott64} as {\it divisibility} of a group.
\begin{lemma}
\label{lemma0}
If $\k$ is a field of zero characteristic
then all the groups $\J_n(\k) , \ \ n = 1, 2, ...$ are complete,
moreover the operation of taking the root is uniquely defined.
\end{lemma}
The proof is easy if we consider the exponential map and in fact it
can be understood from section 4 of \cite{J54}.
We should only note that the group $\J_n(\Q)$ is actually the Malcev completion
of $\J_n(\Z)$.
In other words $\J_n(\Q)$ is the minimal complete group containing $\J_n(\Z)$.
\begin{lemma}
\label{lemmator}
Let $G$ be a compact connected Lie group and
$\phi : \J_n(\Q) \longrightarrow G$ be a continuous homomorphism.
Then $Im(\phi)$ lies in a maximal torus of $G$.
\end{lemma}

\begin{proof}
Without loss of generality we may assume $G = {\bf U}(k)$ for some $k$.
Let $H=\overline{\phi(\J_n(\Q))}$ and let
$H_0$ be the connected component of the identity of this group.
Since $H$ is a compact Lie group then the factor group $H{\big/}H_0$ is finite.

Consider the sequence of homomorphisms
$$
\J(\Q) \mathop{\longrightarrow}\limits^{\phi} H \longrightarrow H{\big/}H_0 .
$$
The above composition is trivial since $\J_n(\Q)$ is a complete group,
so $H$ is a connected Lie subgroup of the unitary group ${\bf U}(k)$.
Since there are no small subgroups in Lie groups ~\cite[p. 107]{MZ}
then there exists a number $m$ such that $\J_m(\Q) \subset \ker \phi$.
Thus the homomorphism $\phi$ factors through the homomorphism
$\hat{\phi}: \J_n^m(\Q) \longrightarrow H$,
so this implies the nilpotency of $H$.
Applying the Lie theorem~\cite[p. 54]{OV90} to $H$
(sometimes this result is known
as the Kolchin-Malcev theorem \cite{KM85} and as the Lie-Kolchin theorem as well \cite{WW79}),
we obtain that all the matrices of $H$ have a common eigenvector.
This together with the unitary property imply

$$
H \subset {\bf U}(1) \times {\bf U}(k-1) \subset {\bf U}(k).
$$
Inductively applying the above Lie theorem $k$ times we obtain
that $H$ lies in a maximal torus of ${\bf U}(k)$.
\end{proof}
By \cite[theorem 3.4]{J54} for a field $\k$ of zero characteristic the equality
$$
\J_{2n+2}(\k) = \overline{[\J_n(\k), \J_n(\k)]}
$$
holds.
Hence we immediately obtain
\begin{cor}
\label{corker}
Under the hypothesis of Lemma~\ref{lemmator} the inclusion
$\J_{2n+2}(\Q) \subset \ker \phi$
is fulfilled.
\end{cor}
Now we turn to the proof of Theorem~\ref{teoQ}.
Let $U \subset X$ be a compact neighborhood of the identity.
Since $\phi$ is continuous there exists $n$ such that
$$
\phi\big{(}\J_n(\Q)\big{)} \subset U.
$$
Thus we have the mapping $\phi : \J_n(\Q) \longrightarrow G'$
where $G' = \overline{\phi\big{(}\J_n(\Q)\big{)}}$ is a compact group.
According to Pontryagin's theorem ~\cite[theorem 68]{Pont73})
$G$ is developed into an inverse sequence of compact Lie groups :
$$
G_1 \mathop{\longleftarrow}\limits^{\pi^2_1}
G_2 \mathop{\longleftarrow}\limits^{\pi^3_2} ...
\mathop{\longleftarrow}\limits^{\pi^k_{k-1}}
G_k \mathop{\longleftarrow}\limits^{\pi^{k+1}_k} ...
\longleftarrow G' .
\eqno(3)
$$
Let $\pi_k : G' \longrightarrow G_k $ be the projection of the limit group on
the $k$-th term of (3).
Let $\phi_k = \pi_k \circ \phi , \ \ k = 1, 2, ...$ .
By applying Corollary~\ref{corker} to each of the homomorphisms
$\phi_k: \J_n(\Q) \longrightarrow G_k$
we get $\J_{2n+2}(\Q) \subset \ker \phi_k$ for all $k$.
This implies
$$
\J_{2n+2}(\Q) \subset \ker \phi ,
\eqno(4)
$$
so the first part of the theorem has been proved.

In the case of a compact group $G$ we can set $n = 0$ in (4),
which finishes the proof.

\begin{remark}
\label{rkuplot}
If a topological ring $\k$ admits a continuous bijection into a compact ring $\k'$
then the group $\J(\k)$ evidently admits a continuous bijection into the compact group $\J(\k')$.
If we pass from the compact case to the locally compact one we have a quite different situation.
The ring of rational numbers $\Q$ supplied, for example, with the discrete topology
admits continuous bijections into locally compact rings, for example, into the ring
of $p$-adic numbers $\Q_{(p)}$.
Theorem~\ref{teoQ} shows that if we pass from $\Q$
to the group $\J(\Q)$ the above property completely disappears.
The latter group cannot admit any
continuous bijection into any locally compact group.
\end{remark}

\section{Compressible topological groups}

\begin{definition}
A topological group $G$ is called {\it compressible}
if for any neighborhood of unity $e \in U \subset G$ there exists an imbedding
$h_U : G \longrightarrow G$ such that $\mbox{Im}h_U \subset U$.
The corresponding embedding $h_U$ is also called {\it compression}.
\end{definition}

\begin{ex}
\label{ex1}
Let $H$ be a topological group and let ${\bf G} = H^{\aleph_0}$
be a countable direct product.
It is easy to see that the group ${\bf G}$ is compressible.

\noindent
All right shifts
$$
h_n(g_1, g_2, ...) = (e, e, ..., e, g_1, g_2, ...) , \ \ n = 1, 2, ... \ \ ,
$$
where $e$ is the identity of $G$ and $g_1$ in the right side is situated on the $(n+1)$-th position,
can be taken as corresponding compressions.

The situation can be radically different in the case of limit groups of inverse sequences
of topological groups. In particular, the $p$-adic solenoids are not compressible
groups.
\end{ex}

\begin{ex}
\label{ex2}
Let's consider the group $G = Homeo_+[0, 1]$ of the homeomorphisms of the interval $[0, 1]$
preserving the endpoints.
We suppose $G$ is supplied with the natural topology of a metric space homeomorphism group.

Let $f(x) \in G$, we assume
$$
h_n\big{(}f\big{)}(x) = {1 \over n}f(nx) , \ \ 0 \leq x \leq {1 \over n} ;
\hskip7pt
h_n\big{(}f\big{)}(x) = x , \ \ {1 \over n} \leq x \leq 1.
$$
It is clear that $h_n$ is an embedding of $G$ into itself for any $n$.
Moreover for an arbitrary neighborhood $U \subset G$ of the identity $e = x$
there exists $n$ such that $\mbox{Im}h_n \subset U$.
Thus the mappings $h_n, \ \ n = 1, 2, ...$ form a system of compressions on $G$.
Note that the Thompson group $F$ \cite{CFP96} is naturally embedded into
$Homeo_+[0, 1]$, and this embedding is compatible with the compressions
$h_{2^k}, \ \ k = 1, 2, ...$.
Thus $h_{2^k}, \ \ k = 1, 2, ...$ are the compressions of the group $F$
if it is supplied not with the discrete topology but with
the topology induced by the embedding $F \subset Homeo_+[0, 1]$.
\end{ex}
Compressibility turns out to be an obstruction for a group to be embedded into locally compact groups.
\begin{prop}
\label{propavto}
A compressible group containing an infinite discrete subgroup
cannot be embedded into any locally compact group.
\end{prop}

\begin{proof}
Suppose the contrary, let $G \subset K$ where $G$ is a compressible group,
and $K$ is locally compact. Consider an arbitrary compact neighborhood of identity
$V \subset K$, and let $U = V \cap G$ be the corresponding neighborhood of identity in $G$.
Consider a compression $h_U: G \longrightarrow U \subset V$,
it induces an imbedding $h_U: G \longrightarrow \overline{h_U(G)}$
where the group $\overline{h_U(G)}$ is compact.
This contradicts Corollary~\ref{corcompactgroup}
since $G$ contains an infinite discrete subgroup.
\end{proof}
The homeomorphism group $G$ of example~\ref{ex2} contains infinite discrete subgroups,
the subgroup generated by the mapping $f(x) = x^2$, for example, is of such a type.
So the group $G$ cannot be embedded into a locally compact group.
\begin{theorem}
\label{teoavtodim}
Any compressible compact group is zero-dimensional or infinite-dimensional.
\end{theorem}

\begin{proof}
Let $G$ be a compressible compact group of finite positive dimension $n=\dim G>0$.
By Pontryagin theorem~\cite[theorem 69]{Pont73}
there exists a zero-dimensional normal subgroup $N$ in $G$ such that
the factor group $G/N$ is a Lie group.
Let $p_N\colon G\to G/N$ be a projection.
Since Lie groups have no small subgroups
there exists a neighborhood $V$ of the identity $p_N(e)$ in $G/N$
which does not contain any nontrivial subgroups,.
Consider the neighborhood of identity $U=p_N^{-1}(V)$ in $G$ and let
$h_U : G \longrightarrow G$ be a corresponding compression.
The subgroup $H=p_N\bigl(h_U(G)\bigr)$ is in $V$ therefore it is trivial.
This implies $h_U(G)\subset N$.
The above inclusion contradicts the dimension monotonicity principle.
So the group $G$ cannot be of finite positive dimension.
\end{proof}
The compactness hypothesis in Theorem~\ref{teoavtodim} is significant.
\begin{ex}
\label{exn}
For any natural number $n$ there exists an algebraically complete, connected, locally connected,
separable, metrizable, compressible group $G_n$, such that $\dim G_n=n$.
In fact, for any $n \in \N$ Keesling~\cite{Kees85} constructed a connected, locally connected,
algebraically complete group $K_n\subset \mathbb R^{n+1}$, such that $\dim K_n=\dim (K_n)^{\aleph_0}=n$.
According to Example ~\ref{ex1} $G_n=(K_n)^{\aleph_0}$ is a group as desired.
\end{ex}
Similar examples can be also constructed in the class of more special groups like
substitution groups $\J(\k)$, see Proposition~\ref{Erd} below.

\section{Dilation  and compressions on the group $\J(\k)$}

For an arbitrary ring $\k$ the group $\J(\k)$ admits two important topological endomorphisms
which we describe in this section.
For any element $f(x) \in \J(\k)$ defined by Formula (1) and for any $t \in \k$
we set
$$
\d_t(f) = {1 \over t}f(tx) = x(1 + t\alpha_1x + t^2\alpha_2x^2 + ...).
\eqno(5)
$$
The second part of equality (5) correctly defines $\d_t$ for all $t \in \k$,
and the first part of (5), which should be understood formally, implies that
$\d_t$ is an endomorphism of the group $\J(\k)$.
If $t$ is not a zero divisor in $\k$ then $\d_t$ is an algebraic monomorphism.
Furthermore if $\k$ is discrete (just the case we are interested in) then
$\d_t$ is an embedding of $\J(\k)$ into itself.
Let the topological endomorphism $\d_t$ introduced in (5) be
called a {\it dilation} of $\J(\k)$ corresponding to a parameter $t \in \k$.
For any $t,r \in \k$, (5) directly implies $\d_t \circ \d_r = \d_{tr}$.
Thus, if $t$ is invertible in $\k$ then $\d_t$ is an automorphism of $\J(\k)$, and
$\d_t$ realizes an embedding of the (multiplicative) group of the invertible elements
of the ring $\k$ into the automorphism group $Aut(\J(\k))$.

To define the other special mapping of $\J(\k)$ into itself we rewrite (1) as follows
$$
f(x) = x(1 + xh(x)),
$$
where $h(x) = \alpha_1 + \alpha_2x + ...$ is
a formal power series with coefficients in $\k$.
Further on, for any natural integer $s$ we set
$$
\Theta_s(f) = x\Big{(}1 + s^2x^sh(s^2x^s)\Big{)}^{1 \over s} ,
\eqno(6)
$$
where the right hand side means a subsequent binomial series expansion
together with the raising of $h(s^2x^s)$ to powers.
To prove the correctness of (6) we use the following fact of elementary analysis.

\begin{lemma}
\label{lemmaanalyse}
For any natural integer $s$
all positive index Taylor coefficients of the function
$$u(z) = \Big{(}1 + s^2z\Big{)}^{1 \over s}$$
are integer and divisible by $s$.
\end{lemma}

\begin{proof}
Let
$$
u(z) = 1 + \b_1z + \b_2z^2 + ...
$$
be the expansion of $u(z)$ at zero.
The integrality of $\b_k, \ \ k = 1, 2, ...$
and the divisibility of these coefficients by $s$ will be proved by induction.
By expanding $u(z)$ into a binomial series we obtain $\b_1 = s$.
Suppose that the statement is proved for all $\b_k, \ k < n$,
now we prove it for $\b_n$. Setting $u_k(z) = \b_1z + \b_2z^2 + ... + \b_kz^k$ we obtain
$$
u(z)^s = \big{(}1 + u_{n-1}(z) + \b_nz^n + \O(x^{n+1}) \big{)}^s =
\big{(}1 + u_{n-1}(z) + \b_nz^n \big{)}^s + \O(x^{n+1})  =
$$
$$
\big{(}1 + u_{n-1}(z)\big{)}^s + s\b_nz^n  + \O(x^{n+1})  =
1 + su_{n-1}(z) + \mathop{\sum}\limits_{k=2}^sC^k_s(u_{n-1}(z))^k +
s\b_nz^n + \O(x^{n+1}) .
\eqno(7)
$$
Since all the coefficients of the polynomial $u_{n-1}(z)$ are divisible by $s$,
then the coefficients of
$\mathop{\sum}\limits_{k=2}^sC^k_s(u_{n-1}(z))^k$ are divisible by $s^2$.
By setting the sum of coefficients attached to $z^n$ in the right hand side of (7) equal to zero,
and according to the functional equation $u(z)^s = 1 + s^2z$ we obtain, for $s\b_n$,
an integer expression divisible by $s^2$.
\end{proof}
\begin{remark}
\label{rkeisenstein}
The above lemma is a particular case of a general result obtained by Eisenstein,
see \cite[sec. 8, chap. 3]{PS38}.
If a series $u(z) = \b_1z + \b_2z^2 + ...$ with rational coefficients is
an algebraic function in $z$ then there exists a natural number $T$ such that
the series $u(Tz) = \b_1Tz + \b_2T^2z^2 + ...$ has integer coefficients.
The least $T$ which verify the above property is the so called Eisenstein index.
It can be expressed by prime divisors of coefficients of the equation for $u(z)$.
Sometimes this expression can be rather complex which we can see in the simplest
case $\big{(}u(z) + 1\big{)}^s = 1 + z$ we are interested in.
Lemma~\ref{lemmaanalyse} shows that in the situation under consideration we can choose $T = s^2$.
The quantity $T = s^2$ exceeds, as a rule, the corresponding Eisenstein index
but it is quite enough for our purposes (cf. \cite[sec. 8, chap. 3]{PS38}).
\end{remark}
It is clear that for an arbitrary natural $s$ all power series having the structure
$$
f(x) = x(1 + \alpha_1x^s + \alpha_2x^{2s} + ... + \alpha_nx^{ns}+ ...),
\hskip7pt \alpha_n \in {\bf k}
$$
form a closed subgroup in $\J(\k)$, which we denote as $\J_{(s)}(\k)$ .
For any $s$ we have the obvious inclusions $\J_{(s)}(\k) \subset \J_{s-1}(\k)$.
The subgroups $\J_{(s)}(\k)$ were closely studied in the case of a field $\k$
of a positive characteristic, see, for example, \cite{C99}.
Some results obtained in this article are not related to phenomena of finite characteristic
but they are of general nature and they are subject to the compressibility
scheme mentioned above.

\begin{prop}
\label{proptheta}
For an arbitrary natural $s$ the mapping $\Theta_s$ is correctly defined, and it is
a homomorphism of the group $\J(\k)$ into $\J_{(s)}(\k)$.
If $s = s1,\ \ 1 \in \k$ is not a zero divisor in $\k$ then $\Theta_s$ is a monomorphism;
moreover if $\k$ is discrete then $\Theta_s$ is an embedding.
If $s$ is invertible in $\k$ then $\Theta_s$ is an isomorphism of $\J(\k)$
onto the group $\J_{(s)}(\k)$.
\end{prop}

\begin{proof}
Lemma~\ref{lemmaanalyse} immediately implies the correctness of the definition of $\Theta_s$.
If $\k$ is a discrete ring the continuity of the above mapping is evident,
in the general case we proceed as follows.
If $f(x) = x(1 + xh(x) + \O(x^{k+2}))$, where $h(x) \in \k[z]$
is a polynomial of degree $k$, then (6) implies
$$
\Theta_s(f(x)) = x(1 + H(x^s) + \O(x^{s(k+2)})),
$$
where $H(z) \in \k[z]$ is a polynomial of degree $k+1$ and moreover
$$
H(z) + 1 = \big{(}1 + s^2zh(s^2z)\big{)}^{1 \over s} \mod z^{k+2}.
$$
So the coefficients of $H(z)$ are expressed by the coefficients of $h(x)$
by means of universal integer polynomials.
This implies the continuity of $\Theta_s$.
The formal presentation $\Theta_s = \theta_s \circ \d_{s^2}$ implies that $\Theta_s$
is a homomorphism, here $\d_{s^2}$ is the dilation with the parameter $s^2 = s^21$ and
$$
\theta_s(f(x)) = f(x^s)^{1 \over s}.
\eqno(8)
$$
If $f(x) = x(1 + \a_nx^n + ...)$, where $\a_n \not= 0$, then we have from (6)
$$
\Theta_s(f) = x\Big{(}1 + s^{2n-1}\a_nx^{ns} + \O(x^{ns + 1})\Big{)}.
$$
This implies that $\Theta_s$ is a monomorphism in case when $s$ is not a zero divisor in $\k$.
Finally, we note that for an arbitrary natural $n$ formula (6) implies the existence of
the induced homomorphisms
$\Theta_s^n : \J^{n-1}(\k) \longrightarrow \J^{sn-1}(\k)$.
By virtue of all the above arguments $\Theta_s^n$ is a monomorphism if $s$
is not a zero divisor in $\k$.
It remains to note that the discreteness of $\k$ implies the discreteness of the groups
$\J^n(\k), \ \ n = 1, 2, ...$, i.e. the monomorphisms $\Theta_s^n$ are topological embeddings.
It is easy to see that
$$
\Theta_s = \mathop{\lim}\limits_{\mathop{\longleftarrow}\limits_n} \Theta_s^n,
$$
which implies that $\Theta_s$ is an embedding.
Finally, let $s$ be invertible in $\k$, and let
$$
\hat{\theta_s}(f(x)) = \Big{(}f(x^{1 \over s})\Big{)}^s,
\eqno(9)
$$
so the mapping
$$
\hat{\Theta}_s =
\d_{s^{-2}} \circ \hat{\theta_s} :\J_{(s)}(\k) \longrightarrow \J(\k),
$$
is correctly defined.
The continuity and monomorphism of $\hat{\Theta}_s$ can be verified in the same way
as in the case of $\Theta_s$ considered above.
Formulae (8) and (9) immediately imply
the equality $\hat{\Theta}_s \circ \Theta_s = \mbox{id}$.
The proof is complete.
\end{proof}

\begin{definition}
We call the endomorphism $\Theta_s$ introduced in (6) a compression in
the group $\J(\k)$ with the coefficient $s$.
\end{definition}

\begin{remark}
\label{rkalglie}
If a ring $\k$ contains the field of rational numbers $\Q$ we can eliminate the dilation
by simply setting $\Theta_s = \theta_s$, where $\theta_s$ is defined by (8).
In case of the group $\J(\R)$ the compression $\theta_s$ induces an endomorphism
of the corresponding Lie algebra denoted as $\theta_s^*$. In the canonical basis
$$L_1 = \{e_n, \ \ n = 1, 2, ...; \ \ [e_n, e_m] = (m-n)e_{n+m} \}$$
it has a simple form
$\theta_s^*(e_n) = {1 \over s}e_{sn} ,\ \ n = 1, 2, ...$.
The mapping $\theta_s^*$ is induced by the $s$-sheeted covering of the circle
if we interpret the elements of $L_1$ as polynomial vector fields on it.
The mapping $\theta_s^*$ is universal, i.e. it is defined on all the algebras
$L_k , k = 0, 1, 2, ...$ where the numeration of basis vectors is carried out for $n \geq k$,
and it is also defined on the complete Witt algebra $L$, where $n$ runs all the integers.
The analogs of the endomorphisms $\theta_s^*$ also exist
in an arbitrary finite characteristic $p$ when $s$ is relatively prime with $p$.
\end{remark}

\begin{prop}
\label{Erd}
There exists such a topological ring with unity $\k$ that the substitution group
$\J(\k)$ is a one-dimensional separable metrizable compressible group.
\end{prop}

\begin{proof}
Consider the Erd\"os space $E$ which is the subset of the Hilbert space $\ell_2$,
that consists of sequences of rational numbers.

Endow $E$ with a ring structure as follows.
Let elements of $\ell_2$ have coordinates
${\bf x} = (x_0, x_1, x_2, ...)$, write down this element ${\bf x}$ in the form
${\bf x} = (x_0, \bar{x})$, where $\bar{x}$ contains all the coordinates
except $x_0$.
To define a multiplication on $\ell_2$ we set
$$
{\bf x}{\bf y} = (x_0y_0, x_0\bar{y} + y_0\bar{x} + \bar{x}\bar{y}),
$$
for two vectors ${\bf x} = (x_0, \bar{x})$ and ${\bf y} = (y_0, \bar{y})$,
where $\bar{x}\bar{y}$ is the coordinate-wise multiplication.
It is clear that $\ell_2$ supplied with such multiplication becomes an algebra
with unity $ {\bf 1} = (1, \bar{0})$, and so $E$ becomes a sub-ring with unity (and even a $\Q$-algebra).
We take the topological ring $E$ just defined as the base ring $\k$.
The topological space $E$ is one-dimensional and all its finite powers
are homeomorphic to itself, so they are also one-dimensional.
Its countable power $E^{\aleph_0}$ proves to be also one-dimensional
~\cite[examples 1.2.15 and 1.5.17]{Eng}.
Thus the substitution group $\J(E)$ yields a desired example.
Since $E$ is a $\Q$-algebra then homomorphisms (8)
form a system of compressions on $\J(E)$.
\end{proof}

\bigskip

{

\end{document}
\parindent =0.7truecm

\begin{thebibliography}{150}







\bibitem{BB08}I.K. Babenko, S.A. Bogatyi,
{\em On the group of substitutions of formal integer power series},
Izvestia Mathematics Vol. 72 2008, n 2, pp. 39--64.
%
\bibitem{BB09}
I.K. Babenko, S.A. Bogatyi,
{\em On the amenability of the group of substitutions of formal power series}.
preprint, 2009, to appear in Izvestia Mathematics.

\bibitem{C99}
R.D. Camina,
{\em Some natural subgroups of the Nottingham Group}.
Proceedings of the Edinburg Math. Soc., 1999, V 42, 333-339.

\bibitem{C00}
R.D. Camina,
{\em The Nottingham group}.
In New Horizons in Pro-$p$ Groups,
M.P.F. du Sautoy, D. Segal, and A. Shalev, Eds., Birkh\"auser, Basel, 2000.

\bibitem{CFP96}
J.W. Cannon, W.J. Floyd, W.R. Parry,
{\em Introductory notes on Richard Thompson's groups}.
l'Enseignement Math. (2), 1996, V 42, 215-256.

\bibitem{Eng}
R. Engelking,
{\em Dimension theory}.
PWN--Polish, Warszawa 1978.

\bibitem{J54}
S.A. Jennings,
{\em Substitution groups of formal power series}.
Canad. J. of Math., 1954, V 6, 325--340.

\bibitem{KM85}
M. Kargapolov, Iou. Merzliakov,
{\em \'El\'ements de la th\'eorie des groupes}.
\'Edition MIR, Moscou 1985.

\bibitem{Kees85}
J. Keesling,
{\em An $n$-dimensional subgroup of $\mathbb R^{n+1}$}.
Proc. Amer. Math. Soc., 1985, V 95, n 1, 106-108.

\bibitem{MZ}
D. Montgomery, L. Zippin
{\em Topological transformation groups}.
Interscience publ., New York--London 1965.

\bibitem{OV90}
A.L. Onishik, E.B. Vinberg,
{\em Lie groups and algebraic groups}.
Springer Verlag, 1990.

\bibitem{PS38}
G. Polya, G. Szeg\"o,
{\em Problems and Theorems in Analysis}, Vol. 2. (Die Grundlehren n 216)
Springer Verlag,
Berlin 1976.

\bibitem{Pont73}
L. S. Pontryagin,
{\em Topological Groups. (sec. edition)}
Gordon and Breach Science Pub., 1966.

\bibitem{Scott64}
W. R. Scott,
{\em Group Theory}.
Prentice-Hall, INC 1964.

\bibitem{WW79}
W.C. Waterhouse,
{\em Introduction to Affine Group Schemes}.
Springer Verlag 1979.

\end{thebibliography}
